\documentclass[12pt,a4paper]{amsart}

\usepackage{amsmath,amssymb,amsthm,mathrsfs}
\usepackage[T1]{fontenc}
\usepackage{lmodern}
\usepackage[margin=1in]{geometry}
\usepackage{microtype}
\usepackage{cite}
\usepackage[colorlinks=true,linkcolor=blue,citecolor=blue,urlcolor=blue]{hyperref}

\newtheorem{theorem}{Theorem}[section]
\newtheorem{proposition}[theorem]{Proposition}
\newtheorem{lemma}[theorem]{Lemma}

\newtheorem*{proposition*}{Proposition}
\newtheorem*{corollary*}{Corollary}
\numberwithin{equation}{section}

\newcommand{\F}{\mathbb F_2}
\newcommand{\Z}{\mathbb Z}
\newcommand{\G}{G_2}
\newcommand{\BG}{B\G}
\newcommand{\Ek}{\mathscr E_k}
\newcommand{\Sq}{\operatorname{Sq}}
\newcommand{\Map}{\operatorname{Map}}
\newcommand{\Hom}{\operatorname{Hom}}
\newcommand{\Ext}{\operatorname{Ext}}

\newcommand{\id}{\operatorname{id}}
\newcommand{\eps}{\varepsilon}
\newcommand{\bk}{\overline{k}}
\newcommand{\doi}[1]{\href{https://doi.org/#1}{\nolinkurl{#1}}}
\def\DD{D\kern-.7em\raise0.4ex\hbox{\char '55}\kern.33em}

\title[Low-degree cohomology of $G_2$-gauge groups]
{Low-degree mod $2$ cohomology\\ of classifying spaces of $G_2$-gauge groups}

\author{Ph\'uc V\~o \DD\d{\u a}ng}

\address{Department of Mathematics, FPT University, An Phu Thinh New Urban Area, Quy Nhon, Vietnam}
\email{dangphuc150488@gmail.com}
\thanks{ORCID: \url{https://orcid.org/0000-0002-6885-3996}}

\subjclass[2020]{Primary 55R35; Secondary 55S10, 55P35, 55T10}
\keywords{Gauge groups, classifying spaces, Steenrod squares, Serre spectral sequence}

\begin{document}
\begin{abstract}
Let $\mathcal G_k$ be the gauge group of the principal $G_2$-bundle over $S^4$ classified by $k\in\pi_4(BG_2)\cong\Z$. In this paper, we determine the first positive fibre row in the mod $2$ Serre spectral sequence of $\Omega^3_0G_2\to B\mathcal G_k\xrightarrow{p_k}BG_2$. The fibre cohomology vanishes in degrees $1$ through $4$ and has a unique nonzero class $u_5$ in degree $5$. We prove that its transgression is $d_6(u_5)=(k\bmod 2)x_6$. The proof uses the Stiefel--Whitney classes of the seven-dimensional representation and maps of evaluation fibrations induced by the pinch map of $S^4$. The same comparison maps show that the possible incoming differential $d_4:H^8(\Omega^3_0G_2;\mathbb F_2)\to\mathbb F_2\{x_4u_5\}$ vanishes. Consequently, the classes $u_5$ and $x_4u_5$ survive precisely when $k$ is even. In degree six, a nonzero linear functional $\lambda:H^6(\Omega^3_0G_2;\mathbb F_2)\to\mathbb F_2$ determines the next transgression by $d_7(z)=(k\bmod2)\lambda(z)x_7$. Both $d_7$ maps to the degree-seven and degree-eleven base classes vanish for even $k$ and are surjective for odd $k$; in particular, $p_k^*x_7\ne0$ exactly when $k$ is even. We derive further forced vanishings, cohomological consequences, and Steenrod-square relations, and delimit the remaining differentials with source of total degree at most $10$. The calculation is independent of the exact order of the fundamental Samelson product.
\end{abstract}
\maketitle

\section{Introduction}

Let $\G$ be the compact simply connected exceptional Lie group of rank two. Choose a generator of $\pi_4(\BG)\cong\Z$, and let $\mathcal G_k$ denote the group of bundle automorphisms covering the identity of the principal $\G$-bundle over $S^4$ classified by $k$, with the compact-open topology. Mapping spaces carry the compactly generated compact-open topology. The integer $k$ will always refer to this homotopy-theoretic normalization. All cohomology groups without an indicated coefficient group have coefficients in $\F$.

Gottlieb's mapping-space description of gauge-group classifying spaces \cite{Gottlieb} gives a homotopy equivalence
\[
\Ek:=\Map_k(S^4,\BG)\simeq B\mathcal G_k.
\]
Fix basepoints of $S^4$ and $\BG$. Evaluation at the basepoint of $S^4$ defines a fibration
\begin{equation}\label{eq:evaluation-fibration}
F_k:=\Map_*(S^4,\BG)_k\longrightarrow\Ek
\xrightarrow{\ p_k\ }\BG.
\end{equation}
The components $F_k$ of $\Omega^4\BG$ are homotopy equivalent by translation. We write $F=F_0\simeq\Omega^3_0\G$ and $M=H^*(F;\F)$, with $M^j=H^j(F;\F)$ for every $j\ge0$. These identifications will be made explicit when comparing different components.

The homotopy classification of gauge groups over spheres is closely related to Samelson products. For $\G$, Kishimoto, Theriault and Tsutaya \cite[Theorem~1.1]{KTT} proved that the fundamental Samelson product has order $84$ or $168$. In particular, its order divides $168$, and its $2$-primary order divides $8$. These bounds provide a comparison with cohomological invariants, but they do not by themselves identify evaluation fibrations over $\BG$. The exact $2$-local classification in terms of $\gcd(k,8)$ is treated in Kameko's preprint \cite[Proposition~1.2]{Kameko}. The cohomological transgression arguments below are independent of that classification and instead use explicit maps of evaluation fibrations over $\BG$.

For $\mathrm{SU}(2)$-bundles over certain four-dimensional complexes, Masbaum \cite{Masbaum} studied the cohomology of gauge-group classifying spaces. Choi \cite{Choi} computed the homology of the $\G$-gauge groups $\mathcal G_k$ over $S^4$ using Eilenberg--Moore and Serre spectral sequences. Because $\mathcal G_k\simeq\Omega B\mathcal G_k$, a homology calculation for the gauge group does not by itself identify the cohomology algebra of its classifying space with Steenrod action.

In this paper, we determine the low-degree mod $2$ cohomology of the classifying spaces $B\mathcal G_k\simeq\Ek$, explicitly describing their cup products and unstable Steenrod operations. Our specific calculation determines the first Serre transgression in the evaluation fibration \eqref{eq:evaluation-fibration}, the survival of the first positive fibre row, and the next transgression into the degree-seven base class.

Classical homotopy groups imply that $M^j=0$ for $1\le j\le4$ and $M^5=\F\{u_5\}$. Consequently, the first possible transgression has the form $d_6(u_5)=\eps(k)x_6$, where $x_6$ is the degree-six Stiefel--Whitney generator of $H^*(\BG)$. We determine this scalar exactly: it is the reduction of $k$ modulo $2$. Thus the scalar is already $2$-periodic, and it vanishes for every even bundle class. The calculation combines a characteristic-class relation under universal evaluation with additivity induced by the pinch map of $S^4$.

Determining this outgoing differential does not alone determine the fate of its multiples. In total degree at most $10$, the explicit evaluation of the spectral sequence yields the following results:
\begin{itemize}
\item \textbf{The first positive fibre row:} The possible incoming differential $d_4$ from fibre degree $8$ into $x_4u_5$ vanishes. This is proved using multiplicativity for odd bundle classes and the pinch-map comparison for even classes. Consequently, both $u_5$ and $x_4u_5$ survive if and only if $k$ is even.
\item \textbf{The $d_7$ differentials:} The degree-six pinch comparison further shows that the two $d_7$ maps into the base row vanish for even bundle classes and are surjective for odd classes. In particular, $p_k^*x_7$ is nonzero exactly for even $k$.
\item \textbf{Forced vanishings for odd classes:} For odd $k$, the surjectivity of $d_7$ and the derivation rule $d_6(x_6u_5) = x_6^2 \neq 0$ force the targets of the displayed $d_{10}$ and $d_{11}$, as well as the $d_6$ from bidegree $(0,10)$, to vanish.
\item \textbf{Cohomology and Steenrod operations:} The associated graded calculation provides actual low-degree cohomology groups, specifying nonzero products and Steenrod squares, including $\Sq^2b_{4,k}=b_{6,k}$ and $\Sq^1b_{6,k}=b_{7,k}$.
\end{itemize}

The results do not constitute a determination of the entire cohomology algebra, or of the entire unstable module, through degree $10$. Classes in higher fibre degrees can support other differentials. We give an explicit list of their possible source bidegrees and pages, including differentials from total degree $10$ to total degree $11$.

The paper is structured as follows. Section~\ref{sec:inputs} records the classical homotopy and characteristic-class inputs and determines the first fibre class. Section~\ref{sec:transgression} computes its transgression by maps of evaluation fibrations. Section~\ref{sec:incoming} treats incoming differentials and determines the first fibre row at infinity. Section~\ref{sec:consequences} derives the cohomological and Steenrod-square consequences. Finally, Section~\ref{sec:remaining} determines the parity of the $d_7$ maps, proves the resulting forced vanishings, and identifies the remaining possible differentials in the specified range.

\section{Homotopy and cohomological inputs}\label{sec:inputs}

We first separate the bound on the connecting map from the inputs needed for the cohomology calculation. We then fix the characteristic-class generators and determine the connectivity, first cohomology group, and first Bockstein of the fibre.

\subsection{The connecting map}

Let $i_3:S^3\to\G$ represent a generator of $\pi_3(\G)$, and let $\partial_k:\G\to F$ be the connecting map of \eqref{eq:evaluation-fibration}, after translation of its target component. Lang's description of evaluation boundary maps \cite{Lang}, in the gauge-group formulation of \cite{KTT}, identifies its triple adjoint with the Samelson product $\langle k i_3,\id_{\G}\rangle$. Choose the usual compatible sign convention for this identification. Bilinearity gives
\begin{equation}\label{eq:boundary-linearity}
[\partial_k]=k[\partial_1].
\end{equation}
The order statement from \cite[Theorem~1.1]{KTT} therefore yields
\begin{equation}\label{eq:published-annihilation}
168[\partial_1]=0.
\end{equation}

\begin{proposition*}
Let $k,l\in\Z$. After localization at $2$, the connecting maps satisfy $8[\partial_1]=0$. If $k\equiv l\pmod8$, then $[\partial_k]=[\partial_l]$ after localization at $2$.
\end{proposition*}

\begin{proof}
The connected triple loop space $F$ is a group-like, homotopy commutative $H$-space. On its $2$-localization, the $21$-fold power map induces multiplication by $21$ on every positive homotopy group. These groups are $2$-local, so this multiplication is an isomorphism. The power map is therefore a homotopy equivalence. In particular, multiplication by $21$ is injective on homotopy classes of maps into $F_{(2)}$. Applying this fact to \eqref{eq:published-annihilation} gives $8[\partial_1]=0$. Equation~\eqref{eq:boundary-linearity} then gives
\[
[\partial_k]-[\partial_l]=(k-l)[\partial_1]=0
\]
whenever $8$ divides $k-l$.
\end{proof}

\subsection{Characteristic classes}

Let $\rho:\G\to\mathrm{SO}(7)$ be the standard real representation. We choose
\begin{equation}\label{eq:base-ring}
R:=H^*(\BG)=\F[x_4,x_6,x_7],\qquad x_i=w_i(\rho),
\end{equation}
with $|x_i|=i$. The polynomial calculation is classical \cite{Borel}; the realization of these generators as the indicated Stiefel--Whitney classes is also recorded in \cite[Section~6]{BakerBauer}. Write $w_i=w_i(\rho)$ and $w_0=1$. The rank condition gives $w_i=0$ for $i>7$, and \eqref{eq:base-ring} gives $w_1=w_2=w_3=w_5=0$.

The Wu formula, in the convention
\[
\Sq^a w_b=
\sum_{t=0}^{a}\binom{b-a+t-1}{t}w_{a-t}w_{b+t},
\]
is recalled in \cite[p.~146]{MimuraNishimoto}. Its relevant specializations are
\begin{align}
\Sq^1x_4&=0,& \Sq^2x_4&=x_6,
&\Sq^1x_6&=x_7,\label{eq:small-squares}\\
\Sq^4x_6&=x_4x_6.&&&\label{eq:wu-key}
\end{align}
Indeed, before applying the rank and vanishing conditions, the expressions needed here are
\[
\Sq^1w_4=w_1w_4+w_5,\qquad
\Sq^2w_4=w_2w_4+w_6,\qquad
\Sq^1w_6=w_1w_6+w_7,
\]
and
\[
\Sq^4w_6=w_4w_6+w_2w_8+w_{10}.
\]
These identities will be used on cohomology groups of spaces, rather than through an unspecified action on pages of a spectral sequence.

\subsection{The first fibre class}

The required homotopy groups, computed by Mimura \cite[p.~132]{Mimura}, are
\begin{equation}\label{eq:homotopy-groups}
\begin{gathered}
\pi_3(\G)\cong\Z,\qquad
\pi_4(\G)=\pi_5(\G)=0,\\
\pi_6(\G)\cong\Z/3,\qquad
\pi_7(\G)=0,\qquad
\pi_8(\G)\cong\Z/2.
\end{gathered}
\end{equation}

\begin{lemma}\label{lem:fibre-low}
Let $F=\Omega^3_0\G$ up to homotopy. Then $F$ is $2$-connected, $F_{(2)}$ is $4$-connected, and
\[
H_j(F_{(2)};\Z)=0\quad(1\le j\le4),\qquad
H_5(F_{(2)};\Z)\cong\Z/2.
\]
Moreover,
\[
M^0\cong\F,\qquad M^j=0\quad(1\le j\le4),\qquad
M^5\cong\F.
\]
\end{lemma}

\begin{proof}
The suspension--loop adjunction identifies $\pi_i(F)$ with $\pi_{i+3}(\G)$ for every $i\ge1$. Equation~\eqref{eq:homotopy-groups} therefore gives $\pi_1(F)=\pi_2(F)=0$. In particular, $F$ is simply connected and nilpotent. The localization theorem for nilpotent spaces \cite{HiltonMislinRoitberg} identifies
\[
\pi_i(F_{(2)})\cong\pi_i(F)\otimes\Z_{(2)}
\]
and makes $F\to F_{(2)}$ an isomorphism on mod $2$ homology. The same equation now gives
\[
\pi_i(F_{(2)})=0\quad(1\le i\le4),\qquad
\pi_5(F_{(2)})\cong\Z/2.
\]
Apply the integral Hurewicz theorem to the $4$-connected space $F_{(2)}$ \cite{Hatcher}. It gives the asserted integral homology vanishing and an isomorphism $\pi_5(F_{(2)})\cong H_5(F_{(2)};\Z)$.

The universal coefficient sequence for cohomology is
\[
0\longrightarrow
\Ext^1_{\Z}(H_{j-1}(F_{(2)};\Z),\F)
\longrightarrow H^j(F_{(2)};\F)
\longrightarrow
\Hom_{\Z}(H_j(F_{(2)};\Z),\F)
\longrightarrow0.
\]
For $1\le j\le4$ both outer terms vanish; when $j=1$, the left-hand term vanishes because $H_0(F_{(2)};\Z)=\Z$ is free. For $j=5$, the left-hand term is zero and the right-hand term is $\Hom_{\Z}(\Z/2,\F)\cong\F$. The mod $2$ homology equivalence under localization, followed by universal coefficients over the field $\F$, transfers these cohomology calculations to $F$. Connectedness gives $M^0\cong\F$.
\end{proof}

Fix the unique nonzero element $u_5\in M^5$.

\begin{lemma}\label{lem:bockstein}
Let $u_5$ be as above. Then $\Sq^1u_5\ne0$ in $M^6$.
\end{lemma}

\begin{proof}
Set $X=F_{(2)}$. The integral homology calculation in Lemma~\ref{lem:fibre-low} gives natural identifications
\[
H^5(X;\Z/4)\cong\Hom_{\Z}(\Z/2,\Z/4),\qquad
H^5(X;\Z/2)\cong\Hom_{\Z}(\Z/2,\Z/2).
\]
Every homomorphism $\Z/2\to\Z/4$ takes values in $\{0,2\}$. Consequently, reduction modulo $2$ induces the zero map from the first displayed cohomology group to the second. The long exact sequence associated with
\[
0\longrightarrow\Z/2\xrightarrow{\ 1\mapsto2\ }\Z/4
\longrightarrow\Z/2\longrightarrow0
\]
therefore makes its connecting homomorphism $H^5(X;\Z/2)\to H^6(X;\Z/2)$ injective. This connecting homomorphism is $\Sq^1$. It is nonzero on the nonzero degree-five class. Naturality and the mod $2$ cohomology isomorphism induced by $F\to X$ prove the assertion.
\end{proof}

\section{The first transgression}\label{sec:transgression}

The fibre calculation fixes the page on which the first transgression can occur. We compute its dependence on the bundle class using a pinch-map comparison of evaluation fibrations, and determine its value on a generator using \eqref{eq:wu-key}.

\subsection{Spectral-sequence conventions}

Let $E_r^{s,t}(k)$ denote the cohomological Serre spectral sequence of \eqref{eq:evaluation-fibration}. Because $\BG$ is simply connected, its local coefficient system is trivial. After identifying the cohomology of $F_k$ with $M$, the initial page and differentials are
\begin{equation}\label{eq:ss-conventions}
E_2^{s,t}(k)=R^s\otimes M^t,
\qquad d_r:E_r^{s,t}(k)\longrightarrow E_r^{s+r,t-r+1}(k).
\end{equation}
We use multiplicativity and naturality of this spectral sequence in the usual cohomological convention \cite{McCleary}. Its convergence provides a finite decreasing filtration $\mathcal F$ in each degree:
\begin{equation}\label{eq:filtration}
\begin{gathered}
0=\mathcal F^{n+1}H^n(\Ek)\subseteq\mathcal F^nH^n(\Ek)
\subseteq\cdots\subseteq\mathcal F^0H^n(\Ek)=H^n(\Ek),\\
\mathcal F^sH^n(\Ek)/\mathcal F^{s+1}H^n(\Ek)
\cong E_\infty^{s,n-s}(k).
\end{gathered}
\end{equation}
The filtration is multiplicative. In particular, an assertion about $E_\infty$ will be used as an assertion about this associated graded object unless a lift is explicitly constructed.

\begin{lemma}\label{lem:first-transgression}
Let $k\in\Z$. Then $u_5\in E_2^{0,5}(k)$ survives to $E_6$, $E_6^{6,0}(k)=\F\{x_6\}$, and there is a unique $\eps(k)\in\F$ such that
\begin{equation}\label{eq:epsilon-definition}
d_6(u_5)=\eps(k)x_6.
\end{equation}
If this differential vanishes, then $u_5$ is a permanent cycle.
\end{lemma}

\begin{proof}
For $2\le r\le5$, the target of a differential from $(0,5)$ has fibre degree $6-r$, which lies between $1$ and $4$; it vanishes by Lemma~\ref{lem:fibre-low}. Thus $u_5$ survives to the sixth page. For $2\le r\le5$, a differential entering $(6,0)$ would originate at $(6-r,r-1)$. Its fibre degree lies between $1$ and $4$, so this source also vanishes. Base-row classes have no outgoing differentials. Hence $E_6^{6,0}(k)=R^6=\F\{x_6\}$, and the scalar in \eqref{eq:epsilon-definition} is uniquely defined. If $u_5$ survives $d_6$, every subsequent outgoing differential has negative fibre degree in its target. Together with the absence of incoming differentials, this proves permanence.
\end{proof}

\subsection{Pinch-map comparisons}

Choose a based pinch map $\nu:S^4\to S^4\vee S^4$ representing the usual addition on maps out of $S^4$. For integers $k,l$, define
\begin{equation}\label{eq:pinch-comparison}
\begin{split}
\mu_{k,l}:\mathscr E_k\times_{\BG}\mathscr E_l
&\longrightarrow\mathscr E_{k+l},\\
(f,g)&\longmapsto(f\vee g)\circ\nu.
\end{split}
\end{equation}
The two maps in the domain have equal values at the basepoint, so $f\vee g$ is defined. Its component is $k+l$, since the pinch map induces addition in $\pi_4(\BG)$, and evaluation is unchanged. Thus \eqref{eq:pinch-comparison} is a map of fibrations over the identity of $\BG$. The domain fibration has fibre $F_k\times F_l$.

For each integer $k$, choose $f_k\in F_k$, taking $f_0$ to be the constant map, and let $T_k:F\to F_k$ be translation by $f_k$ in the iterated loop space. We use $T_k^*$ to identify $H^*(F_k)$ with $M$. A different choice within the same component gives a homotopic translation. Under these identifications the restrictions of the fibre map of $\mu_{k,l}$ to its two factors induce the identity on $M$. Indeed, homotopy associativity and homotopy commutativity of $\Omega^4\BG$ identify either restriction, after translation, with translation by the product of $f_k$ and $f_l$. That product is in the component of $f_{k+l}$, so translation by it is homotopic to $T_{k+l}$.

\begin{lemma}\label{lem:primitive-comparison}
Let $k,l\in\Z$, let $j$ be an integer with $0<j<10$, and let $z\in M^j$. Under the preceding identifications, the fibre map of $\mu_{k,l}$ induces
\[
z\longmapsto z\otimes1+1\otimes z
\quad\text{in }H^j(F_k\times F_l).
\]
\end{lemma}

\begin{proof}
By Lemma~\ref{lem:fibre-low} and universal coefficients over $\F$, the homology of $F$ also vanishes in degrees $1$ through $4$. Indeed, a nonzero vector space over $\F$ has a nonzero linear functional. In total degree $j<10$, the homology K\"unneth theorem therefore has no summand with two positive degrees: each such degree would be at least $5$. Consequently,
\[
H_j(F_k\times F_l;\F)
\cong H_j(F_k;\F)\oplus H_j(F_l;\F).
\]
Dualizing this direct sum and using the translation identifications gives
\[
H^j(F_k\times F_l)
=(M^j\otimes M^0)\oplus(M^0\otimes M^j).
\]
No finite-dimensionality assumption is needed. The restrictions to the two factors are the projections onto these summands. Both restrictions of the pullback of $z$ equal $z$, by the translation comparison preceding the statement. These restrictions determine the asserted pullback uniquely.
\end{proof}

\begin{proposition}\label{prop:epsilon-additive}
Let $k,l\in\Z$. Then $\eps(k+l)=\eps(k)+\eps(l)$.
\end{proposition}

\begin{proof}
Write $\widehat E_r^{s,t}(k,l)$ for the spectral sequence of the domain fibration in \eqref{eq:pinch-comparison}. Its fibre cohomology vanishes in degrees $1$ through $4$, and its degree-five part has the two classes
\[
u=u_5\otimes1,\qquad v=1\otimes u_5.
\]
The bidegree argument of Lemma~\ref{lem:first-transgression} applies to this fibration as well: both classes survive to page six, and $\widehat E_6^{6,0}(k,l)=\F\{x_6\}$. The projections from the fibre product onto its two factors are maps of fibrations over $\BG$. Naturality of \eqref{eq:epsilon-definition} therefore gives
\[
\widehat d_6(u)=\eps(k)x_6,\qquad
\widehat d_6(v)=\eps(l)x_6.
\]
By Lemma~\ref{lem:primitive-comparison}, the map induced by $\mu_{k,l}$ sends the class $u_5$ in the target fibration to $u+v$, and it fixes $x_6$ on the base. Naturality under this map gives
\[
\eps(k+l)x_6
=\widehat d_6(u+v)
=(\eps(k)+\eps(l))x_6.
\]
The base class is nonzero on this page, so its coefficients agree.
\end{proof}

\subsection{Universal evaluation}

Define $b_{i,k}=p_k^*x_i$ for $i=4,6,7$, let $a\in H^4(S^4)$ be its nonzero element, and let
\[
e_k:S^4\times\Ek\longrightarrow\BG,\qquad (s,f)\longmapsto f(s)
\]
be universal evaluation.

\begin{lemma}\label{lem:universal-evaluation}
Let $k\in\Z$. Then
\begin{equation}\label{eq:evaluation-classes}
e_k^*x_4=1\otimes b_{4,k}+\bk\,a\otimes1,
\qquad
e_k^*x_6=1\otimes b_{6,k},
\end{equation}
where $\bk$ denotes the reduction of $k$ modulo $2$. Moreover, $\bk\,b_{6,k}=0$.
\end{lemma}

\begin{proof}
In total degree $2$, every term of \eqref{eq:ss-conventions} is zero. Thus $H^2(\Ek)=0$. The space $\BG$ is $3$-connected, and integral Hurewicz identifies $\pi_4(\BG)\cong\Z$ with $H_4(\BG;\Z)$. Universal coefficients then identify its unique nonzero mod $2$ cohomology class in degree $4$ with the functional taking a generator to $1$. This class is $x_4$, by \eqref{eq:base-ring}.

The K\"unneth decomposition of $H^4(S^4\times\Ek)$ has only the summands $H^0(S^4)\otimes H^4(\Ek)$ and $H^4(S^4)\otimes H^0(\Ek)$. Restriction to the chosen point of $S^4$ gives $b_{4,k}$ in the first summand. Restriction to $S^4\times\{f\}$, for any $f\in\Ek$, gives $\bk a$ in the second, by the normalization just established. In degree $6$ the only possible additional summand would be $H^4(S^4)\otimes H^2(\Ek)$, which is zero. These observations prove \eqref{eq:evaluation-classes}.

Pull back \eqref{eq:wu-key} along $e_k$ and apply naturality of Steenrod squares. The left-hand side is $1\otimes\Sq^4 b_{6,k}$, while the right-hand side is
\[
1\otimes b_{4,k}b_{6,k}+\bk\,a\otimes b_{6,k}.
\]
Their components in $H^4(S^4)\otimes H^6(\Ek)$ must agree. The component on the left is zero, and therefore $\bk\,a\otimes b_{6,k}=0$. Since $a$ is nonzero over the field $\F$, this is equivalent to the asserted relation.
\end{proof}

\begin{theorem}\label{thm:parity}
Let $k\in\Z$. In the spectral sequence of \eqref{eq:evaluation-fibration},
\[
d_6(u_5)=\bk x_6.
\]
The class $u_5$ is a permanent cycle if and only if $k$ is even.
\end{theorem}

\begin{proof}
For $k=1$, Lemma~\ref{lem:universal-evaluation} gives $p_1^*x_6=0$. The base edge homomorphism is the composite
\[
R^6=E_2^{6,0}(1)\longrightarrow E_\infty^{6,0}(1)
\cong\mathcal F^6H^6(\mathscr E_1)
\lhook\joinrel\longrightarrow H^6(\mathscr E_1).
\]
Here $\mathcal F^7H^6(\mathscr E_1)=0$ by \eqref{eq:filtration}. Thus the vanishing of $p_1^*x_6$ means that $x_6$ does not survive to $E_\infty$. By Lemma~\ref{lem:first-transgression}, it survives to page six. A later differential entering $(6,0)$ would have negative base degree in its source. The only possible differential killing $x_6$ is consequently $d_6(u_5)$, so $\eps(1)=1$.

Proposition~\ref{prop:epsilon-additive} states that $\eps:\Z\to\F$ is additive. Taking $k=l=0$ gives $\eps(0)=0$. Repeated addition gives $\eps(n)=n\eps(1)$ for $n>0$, and $0=\eps(n)+\eps(-n)$ gives the same formula for negative integers. Therefore $\eps(k)=\bk$ for every integer $k$. The final assertion follows from Lemma~\ref{lem:first-transgression}.
\end{proof}

\begin{corollary*}
Let $k,l\in\Z$. If $k\equiv l\pmod2$, then $\eps(k)=\eps(l)$. Consequently, $\eps(k+8)=\eps(k)$ for every $k$, and $\eps(k)=0$ whenever $8$ divides $k$.
\end{corollary*}

\begin{proof}
Each assertion follows by reducing its hypothesis modulo $2$ and applying Theorem~\ref{thm:parity}.
\end{proof}

\section{Incoming differentials and the first fibre row}\label{sec:incoming}

We now determine the fate of the multiples of $u_5$ in total degree at most $10$. This requires an incoming-differential argument in addition to Theorem~\ref{thm:parity}. The pinch comparisons provide that argument without any assumption about the remaining fibre cohomology groups.

In the indicated range, the subspace of $R\cdot u_5$ is
\[
(R\cdot u_5)^{\le10}=\F\{u_5,x_4u_5\},
\]
with bidegrees $(0,5)$ and $(4,5)$. There is no incoming differential to the first of these. An incoming differential to the second has source $(4-r,r+4)$. For $r=2,3$, its base degree is $2,1$, respectively, and those base groups vanish. For $r>4$ the source has negative base degree. The sole remaining possibility is
\begin{equation}\label{eq:incoming-possible}
d_4:E_4^{0,8}(k)\longrightarrow E_4^{4,5}(k).
\end{equation}
The differentials $d_2,d_3$ out of column zero have zero base target in degrees $2,3$, and their possible incoming sources at $(4,5)$ have just been excluded. Their outgoing targets from $(4,5)$ have fibre degrees $4,3$. Hence
\[
E_4^{0,8}(k)=M^8,\qquad E_4^{4,5}(k)=\F\{x_4u_5\}.
\]
Define a linear functional $\delta_k:M^8\to\F$ by $d_4(z)=\delta_k(z)x_4u_5$.

\begin{proposition}\label{prop:incoming-zero}
Let $k\in\Z$. Then $\delta_k=0$.
\end{proposition}

\begin{proof}
Suppose first that $k$ is odd. Theorem~\ref{thm:parity} gives $d_6(u_5)=x_6$. Both $x_4$ and $u_5$ exist on page six. Moreover,
\[
E_6^{10,0}(k)=R^{10}=\F\{x_4x_6\}.
\]
Indeed, an incoming differential to $(10,0)$ before page six would have fibre degree $r-1$ with $1\le r-1\le4$, and these fibre rows vanish. In particular, $x_4x_6$ is nonzero on page six.

If $\delta_k$ were nonzero, \eqref{eq:incoming-possible} would be surjective and $x_4u_5$ would be zero on page five, hence on page six. Multiplicativity on page six would then give
\[
0=d_6(0)=d_6(x_4u_5)
=d_6(x_4)u_5+x_4d_6(u_5)=x_4x_6.
\]
The base class has $d_6(x_4)=0$ because its putative target has negative fibre degree. The resulting equality contradicts the nonvanishing established above. Thus $\delta_k=0$ for odd $k$.

Now suppose that $k$ is even, and apply \eqref{eq:pinch-comparison} with $l=1$. The integer $k+1$ is odd, so the first part of the proof gives $\delta_{k+1}=\delta_1=0$. In the product-fibration spectral sequence, the fourth-page bidegree $(4,5)$ is
\[
\widehat E_4^{4,5}(k,1)
=\F\{x_4(u_5\otimes1),\ x_4(1\otimes u_5)\}.
\]
These classes are independent: this is their K\"unneth decomposition on page two, the possible incoming differentials on pages two and three have zero base source, and the outgoing targets have fibre degree $4$ or $3$.

Take any $z\in M^8$. Lemma~\ref{lem:primitive-comparison} describes its pullback under $\mu_{k,1}$ as $z\otimes1+1\otimes z$. Naturality under this map and under the two projections gives
\begin{align*}
0
&=\mu_{k,1}^*\bigl(d_4^{\,k+1}(z)\bigr)\\
&=\widehat d_4(z\otimes1+1\otimes z)\\
&=\delta_k(z)x_4(u_5\otimes1)
 +\delta_1(z)x_4(1\otimes u_5)\\
&=\delta_k(z)x_4(u_5\otimes1).
\end{align*}
The indicated class is nonzero, so $\delta_k(z)=0$. Since $z$ was arbitrary, $\delta_k=0$ also for even $k$.
\end{proof}

\begin{theorem}\label{thm:first-row}
Let $k\in\Z$. The groups in the first positive fibre row and total degree at most $10$ are
\[
E_\infty^{0,5}(k)\cong E_\infty^{4,5}(k)\cong
\begin{cases}
\F,& k\text{ even},\\
0,& k\text{ odd}.
\end{cases}
\]
For even $k$, these groups are represented by $u_5$ and $x_4u_5$, respectively. For odd $k$, their representatives satisfy $d_6(u_5)=x_6$ and $d_6(x_4u_5)=x_4x_6$.
\end{theorem}

\begin{proof}
The two bidegrees are the only ones in this row in the specified range by \eqref{eq:base-ring}. Theorem~\ref{thm:parity} determines the first. Proposition~\ref{prop:incoming-zero} and the preceding source analysis exclude every incoming differential to the second. Its outgoing differentials on pages two through five have target fibre degrees between $1$ and $4$, and therefore vanish. On page six, multiplicativity and Theorem~\ref{thm:parity} give
\[
d_6(x_4u_5)=x_4d_6(u_5)=\bk x_4x_6.
\]
The target is nonzero before this differential, as established in the proof of Proposition~\ref{prop:incoming-zero}. Thus both classes are killed for odd $k$. For even $k$ both survive page six, and every later outgoing target has negative fibre degree. This proves the assertions.
\end{proof}

\section{Cohomology and Steenrod operations}\label{sec:consequences}

The associated graded calculation now yields actual low-degree cohomology groups and specific nonzero products. We distinguish these conclusions from statements that would require the remaining fibre cohomology or extension data.

\begin{theorem}\label{thm:cohomology}
Let $k\in\Z$. Then
\[
H^0(\Ek)=\F,\qquad
H^j(\Ek)=0\quad(1\le j\le3),\qquad
H^4(\Ek)=\F\{b_{4,k}\},
\]
and
\[
H^5(\Ek)\cong
\begin{cases}
\F,&k\text{ even},\\
0,&k\text{ odd}.
\end{cases}
\]
Furthermore,
\[
b_{6,k}\ne0\quad\Longleftrightarrow\quad k\text{ is even},
\qquad
\Sq^2b_{4,k}=b_{6,k},\qquad
\Sq^1b_{6,k}=b_{7,k}.
\]
If $k$ is odd, then $b_{7,k}=0$.
\end{theorem}

\begin{proof}
In total degrees one through three, all terms of \eqref{eq:ss-conventions} vanish, by \eqref{eq:base-ring} and Lemma~\ref{lem:fibre-low}. In degree four the sole nonzero term is $E_2^{4,0}(k)$, generated by $x_4$. No differential leaves it, and every possible incoming source has fibre degree at most $3$, hence vanishes. The edge homomorphism therefore identifies its image with the entire one-dimensional group $H^4(\Ek)$.

In total degree five, the only possible nonzero term is $(0,5)$. Theorem~\ref{thm:first-row} and the filtration \eqref{eq:filtration} give the asserted value of $H^5(\Ek)$. Connectedness gives the degree-zero assertion.

For the base bidegree $(6,0)$, the proof of Lemma~\ref{lem:first-transgression} excludes all incoming differentials except $d_6(u_5)$. Theorem~\ref{thm:parity} therefore gives $E_\infty^{6,0}(k)=\F\{x_6\}$ for even $k$ and zero for odd $k$. The base edge identifies this term with $\mathcal F^6H^6(\Ek)$, because $\mathcal F^7H^6(\Ek)=0$. This proves the assertion about $b_{6,k}$. The two Steenrod-square identities follow from \eqref{eq:small-squares} by naturality under $p_k$. When $k$ is odd, the second identity and $b_{6,k}=0$ give $b_{7,k}=0$.
\end{proof}

Let $j_k:F_k\to\Ek$ denote the fibre inclusion, and use the translation identifications from Section~\ref{sec:transgression} when writing its restriction homomorphism with target $M$.

\begin{proposition}\label{prop:even-classes}
Let $k\in2\Z$. There is a unique class $v_{5,k}\in H^5(\Ek)$ whose restriction to the fibre is $u_5$. It satisfies
\[
\Sq^1v_{5,k}\ne0,\qquad b_{4,k}v_{5,k}\ne0.
\]
The two classes $b_{6,k}$ and $\Sq^1v_{5,k}$ are linearly independent in $H^6(\Ek)$.
\end{proposition}

\begin{proof}
The edge homomorphism to the fibre has image $E_\infty^{0,5}(k)$. For even $k$, this is all of $M^5$ by Theorem~\ref{thm:first-row}. Theorem~\ref{thm:cohomology} gives $\dim_{\F}H^5(\Ek)=1$, so there is exactly one preimage of $u_5$. Naturality and Lemma~\ref{lem:bockstein} give
\[
j_k^*(\Sq^1v_{5,k})=\Sq^1u_5\ne0.
\]
This proves that $\Sq^1v_{5,k}$ is nonzero.

The associated graded class of $b_{4,k}$ is $x_4$ in bidegree $(4,0)$, while that of $v_{5,k}$ is $u_5$ in bidegree $(0,5)$. The multiplicative filtration makes the associated graded class of their product equal to $x_4u_5$ in bidegree $(4,5)$. Theorem~\ref{thm:first-row} states that this class is nonzero. A zero cohomology class would have zero associated graded class, so $b_{4,k}v_{5,k}\ne0$.

Finally, $j_k^*(b_{6,k})=0$, because the composite $p_kj_k$ is constant. If $\alpha b_{6,k}+\beta\Sq^1v_{5,k}=0$ with $\alpha,\beta\in\F$, restriction to the fibre gives $\beta\Sq^1u_5=0$, so $\beta=0$. Theorem~\ref{thm:cohomology} gives $b_{6,k}\ne0$ for even $k$, which then forces $\alpha=0$.
\end{proof}

In particular, even bundle classes satisfy $\dim_{\F}H^6(\Ek)\ge2$. This is a lower bound, not a computation of the entire degree-six group. The product in Proposition~\ref{prop:even-classes} also gives a specified nonzero class in degree nine.

\section{Higher differentials through total degree ten}
\label{sec:remaining}

The first-row calculation provides the input for the next transgression. We determine the $d_7$ maps by the degree-six pinch comparison and then derive further vanishings from multiplicativity. We begin with the full bidegree restrictions for sources of total degree at most $10$, including targets in degree $11$, since these can remove classes relevant to the degree-ten cohomology.

\begin{proposition}\label{prop:remaining-differentials}
Let $k\in\Z$, $r\ge2$, and $s,t\ge0$ with $s+t\le10$. If $d_r:E_r^{s,t}(k)\to E_r^{s+r,t-r+1}(k)$ is nonzero, then
\begin{equation}\label{eq:remaining-triples}
\begin{split}
(r,s,t)\in\{& (4,0,9),(4,0,10),\\
             & (6,0,5),(6,4,5),(6,0,10),\\
             & (7,0,6),(7,4,6),(8,0,7),\\
             & (10,0,9),(11,0,10)\}.
\end{split}
\end{equation}
\end{proposition}

\begin{proof}
Base-row sources have no outgoing differentials. By Lemma~\ref{lem:fibre-low}, a remaining source has $t\ge5$. Equation~\eqref{eq:base-ring} shows that the only possibilities with $s+t\le10$ are
\[
(s,t)=(0,t)\quad(5\le t\le10),\qquad (s,t)=(4,5),(4,6).
\]
A target has total degree $s+t+1\le11$. The positive degrees in which the base can be nonzero up to degree $11$ are
\[
4,6,7,8,10,11,
\]
represented respectively by $x_4,x_6,x_7,x_4^2,x_4x_6,x_4x_7$.

For a source $(0,t)$, the target has bidegree $(r,t-r+1)$. If its fibre degree is zero, then $r=t+1$. Intersecting the range $5\le t\le10$ with the displayed base degrees gives exactly
\[
(r,t)=(6,5),(7,6),(8,7),(10,9),(11,10).
\]
If the target has positive fibre degree, that degree must be at least $5$. Hence $r\le t-4\le6$. The only possible positive base degrees in this range are $4$ and $6$. For $r=4$ the candidates are $t=8,9,10$; the case $t=8$ vanishes by Proposition~\ref{prop:incoming-zero}. For $r=6$ the only candidate is $t=10$. This gives all the entries in \eqref{eq:remaining-triples} with $s=0$.

For source $(4,5)$, the first-row analysis in Section~\ref{sec:incoming} leaves only $r=6$. For source $(4,6)$, a zero fibre degree in the target forces $r=7$, and a positive one forces $7-r\ge5$, hence $r=2$. The latter candidate vanishes by multiplicativity on the second page. Indeed, every element of $E_2^{4,6}(k)$ has the form $x_4z$ with $z\in M^6$; both $d_2(x_4)$ and $d_2(z)$ vanish, the first because its target has negative fibre degree and the second because $R^2=0$. Thus $d_2(x_4z)=0$. Only $(7,4,6)$ remains. Every possible source has now been considered.
\end{proof}

\begin{proposition}\label{prop:d7-parity}
Let $T_k^*:H^*(F_k)\to M$ be the translation identifications fixed in Section~\ref{sec:transgression}. There exists a nonzero linear functional $\lambda:M^6\to\F$, independent of $k$, such that, for every $k\in\Z$,
\[
E_7^{0,6}(k)=M^6,\qquad
E_7^{7,0}(k)=\F\{x_7\},\qquad
E_7^{11,0}(k)=\F\{x_4x_7\},
\]
and
\begin{equation}\label{eq:d7-parity}
d_7(z)=\bk\lambda(z)x_7\qquad(z\in M^6).
\end{equation}
If $k$ is even, then both maps
\[
d_7:E_7^{0,6}(k)\longrightarrow E_7^{7,0}(k),\qquad
d_7:E_7^{4,6}(k)\longrightarrow E_7^{11,0}(k)
\]
are zero. If $k$ is odd, then both maps are surjective. Moreover,
\begin{equation}\label{eq:x7-parity}
p_k^*x_7\ne0\quad\Longleftrightarrow\quad k\text{ is even}.
\end{equation}
\end{proposition}

\begin{proof}
There is no incoming differential to $(0,6)$. The outgoing targets on pages two and three have base degrees two and three, and the outgoing targets on pages four through six have fibre degrees three, two, and one. All these groups vanish by \eqref{eq:base-ring} and Lemma~\ref{lem:fibre-low}. Hence $E_7^{0,6}(k)=M^6$.

For $m\in\{7,11\}$, an incoming differential to $(m,0)$ on pages two through five has source fibre degree between one and four. On page six its source is $(m-6,5)$, whose base degree is one or five. These sources are all zero, and there is no outgoing differential from the base row. This proves the two asserted seventh-page base-group identifications. In particular, a preceding $d_6$ cannot kill $x_7$, regardless of its value on $u_5$. For each $k$, there is consequently a unique linear functional $\lambda_k:M^6\to\F$ satisfying
\[
d_7(z)=\lambda_k(z)x_7.
\]

Use the fibre-product spectral sequence $\widehat E_r(k,l)$ from the proof of Proposition~\ref{prop:epsilon-additive}. Its fibre has zero cohomology in degrees one through four, and its degree-six group is $M^6\oplus M^6$ by Lemma~\ref{lem:primitive-comparison}. The preceding page arguments therefore apply to this fibration: every degree-six fibre class survives to page seven and $\widehat E_7^{7,0}(k,l)=\F\{x_7\}$. The same lemma identifies the pullback of $z\in M^6$ under $\mu_{k,l}$ with $z\otimes1+1\otimes z$. Naturality under this map and under the two projections from the fibre product gives
\begin{align*}
\lambda_{k+l}(z)x_7
&=\widehat d_7(z\otimes1+1\otimes z)\\
&=(\lambda_k(z)+\lambda_l(z))x_7.
\end{align*}
The target class is nonzero on this page. Thus $\lambda_{k+l}=\lambda_k+\lambda_l$ in $\Hom_{\F}(M^6,\F)$. This additive group has characteristic two, so additivity gives
\begin{equation}\label{eq:lambda-additive}
\lambda_k=\bk\lambda_1\qquad(k\in\Z).
\end{equation}
This includes negative integers, since $0=\lambda_k+\lambda_{-k}$.

Theorem~\ref{thm:cohomology} gives $p_1^*x_7=0$. The base edge in degree seven factors through
\[
E_\infty^{7,0}(1)
\cong\mathcal F^7H^7(\mathscr E_1)
\lhook\joinrel\longrightarrow H^7(\mathscr E_1),
\]
because $\mathcal F^8H^7(\mathscr E_1)=0$ by \eqref{eq:filtration}. Hence $x_7$ cannot survive to infinity. It survives to page seven by the base-group calculation above, and a later incoming differential would have negative base degree in its source. Therefore $d_7:M^6\to\F\{x_7\}$ is nonzero for $k=1$, so $\lambda_1\ne0$. Set $\lambda=\lambda_1$. Equation~\eqref{eq:lambda-additive} proves \eqref{eq:d7-parity}. For even $k$, no differential can kill $x_7$, whereas for odd $k$ the seventh differential is surjective onto it. The same base-edge identification proves \eqref{eq:x7-parity}.

It remains to determine the map from $(4,6)$. Its second-page group is $x_4M^6$. The outgoing $d_2$ is zero by the multiplicativity argument in the proof of Proposition~\ref{prop:remaining-differentials}. Its outgoing targets on pages three through six have fibre degrees four through one, so these maps are also zero. The incoming sources on pages two and three have base degrees two and one. After page four, an incoming source has negative base degree. Thus the only possible incoming differential before page seven is
\[
d_4:E_4^{0,9}(k)\longrightarrow E_4^{4,6}(k).
\]
Consequently, $E_7^{4,6}(k)$ is a quotient of $x_4M^6$, and the multiplication map
\[
q_k:M^6\longrightarrow E_7^{4,6}(k),\qquad z\longmapsto x_4z,
\]
is surjective; it need not be injective. The class $x_4$ is permanent by Theorem~\ref{thm:cohomology}, so multiplicativity gives
\begin{equation}\label{eq:product-d7}
d_7(q_k(z))=x_4d_7(z)=\bk\lambda(z)x_4x_7.
\end{equation}
For even $k$, this is zero for every $z$, and the surjectivity of $q_k$ proves vanishing on the entire source. For odd $k$, choose $z$ with $\lambda(z)=1$. The right-hand side of \eqref{eq:product-d7} is nonzero in $E_7^{11,0}(k)=\F\{x_4x_7\}$. In particular, $q_k(z)$ is nonzero, since a zero class has zero differential. This proves surjectivity onto the second base-row target.
\end{proof}

The functional in Proposition~\ref{prop:d7-parity} is characterized by the transgression for $k=1$. Its values on a chosen basis of $M^6$, including its value on $\Sq^1u_5$, are not required for these parity and surjectivity conclusions.

\begin{proposition}\label{prop:odd-forced-vanishing}
Let $k\in\Z$ be odd. Then
\[
E_7^{10,0}(k)=0,\qquad E_8^{11,0}(k)=0,
\]
and
\[
E_6^{6,5}(k)=\F\{x_6u_5\},\qquad
d_6(x_6u_5)=x_6^2\ne0.
\]
The three differentials
\begin{align*}
d_6&:E_6^{0,10}(k)\longrightarrow E_6^{6,5}(k),\\
d_{10}&:E_{10}^{0,9}(k)\longrightarrow E_{10}^{10,0}(k),\\
d_{11}&:E_{11}^{0,10}(k)\longrightarrow E_{11}^{11,0}(k)
\end{align*}
are zero.
\end{proposition}

\begin{proof}
Theorem~\ref{thm:first-row} identifies the differential from $(4,5)$ as the nonzero map onto $E_6^{10,0}(k)=\F\{x_4x_6\}$. Hence $E_7^{10,0}(k)=0$. Proposition~\ref{prop:d7-parity} gives a surjection onto $E_7^{11,0}(k)$, so $E_8^{11,0}(k)=0$. A zero bidegree remains zero on subsequent pages. Thus the targets of the specified $d_{10}$ and $d_{11}$ are zero, proving the two corresponding vanishing statements.

For $(6,5)$, the incoming $d_2$ has source $(4,6)$ and is zero by the proof of Proposition~\ref{prop:remaining-differentials}. The incoming sources on pages three through five have base degrees three, two, and one, and hence vanish. Every outgoing target before page six has fibre degree between one and four and is zero. It follows that
\[
E_6^{6,5}(k)=E_2^{6,5}(k)=\F\{x_6u_5\}.
\]
No differential before page six can enter $(12,0)$, because each such source has fibre degree between one and four. Therefore $E_6^{12,0}(k)=R^{12}$, and the polynomial class $x_6^2$ is nonzero there. By Theorem~\ref{thm:parity} and multiplicativity,
\[
d_6(x_6u_5)=d_6(x_6)u_5+x_6d_6(u_5)=x_6^2\ne0.
\]
Thus the outgoing differential on the one-dimensional group $E_6^{6,5}(k)$ is injective. The composite
\[
E_6^{0,10}(k)\xrightarrow{\ d_6\ }E_6^{6,5}(k)
\xrightarrow{\ d_6\ }E_6^{12,0}(k)
\]
is zero by $d_6^2=0$. Injectivity of the second map forces the first map to be zero.
\end{proof}

\begin{corollary*}
Let $k\in\Z$, $r\ge2$, and $s,t\ge0$ with $s+t\le10$, and assume that
\[
(r,s,t)\notin\{(6,0,5),(6,4,5),(7,0,6),(7,4,6)\}.
\]
If $d_r:E_r^{s,t}(k)\to E_r^{s+r,t-r+1}(k)$ is nonzero and $k$ is odd, then
\[
(r,s,t)\in\{(4,0,9),(4,0,10),(8,0,7)\}.
\]
If this differential is nonzero and $k$ is even, then
\[
(r,s,t)\in\{(4,0,9),(4,0,10),(6,0,10),
(8,0,7),(10,0,9),(11,0,10)\}.
\]
\end{corollary*}

\begin{proof}
Remove the four excluded triples from the necessary set in Proposition~\ref{prop:remaining-differentials}. For odd $k$, Proposition~\ref{prop:odd-forced-vanishing} additionally excludes $(6,0,10)$, $(10,0,9)$, and $(11,0,10)$. This gives the first set. Without those additional exclusions, the resulting set is exactly the second one.
\end{proof}

\textit{Remark.} The first-row differentials are determined by parity. The same pinch-map comparison in degree six shows that the two $d_7$ maps are zero for even bundle classes and surjective onto their base-row targets for odd bundle classes. In the odd case, the targets of the displayed $d_{10}$ and $d_{11}$ have therefore already vanished, and $d_6^2=0$ forces the $d_6$ from bidegree $(0,10)$ to vanish. The remaining entries are only possible bidegree conditions; their inclusion is not an assertion of nonvanishing.

\section*{Acknowledgment}
The author would like to express his sincere gratitude to the handling editor and the anonymous referees for their careful reading of the manuscript and their valuable suggestions and constructive feedback, which have significantly improved the clarity, rigor, and exposition of this work.

\end{document}